\documentclass[10pt,a4paper]{article}

\usepackage{amsfonts}
\usepackage{amsthm}
\usepackage{amsmath}
\usepackage{amscd}
\usepackage{latexsym}
\usepackage{amssymb}
\usepackage[all]{xy}

\usepackage{blkarray}
\newcommand{\matindex}[1]{\mbox{\scriptsize#1}}

\usepackage{multicol}

\def\RR{\mathbb{R}}
\def\CC{\mathbb{C}}

\def\n{ {\mathcal N} }
\def\m{ {\mathcal M} }
\def\h{ {\mathcal H} }

\def\u{ {\mathcal U} }

\def\t{ {\mathcal T} }
\def\q{ {\mathcal Q} }
\def\s{ {\mathcal S} }
\def\e{ {\mathcal E} }
\def\p{ {\mathcal P} }

\def\k{ {\mathcal K} }
\def\f{ {\mathcal F} }

\def\noi{\noindent}

\def\g1{ \mathfrak{g}_1  }

\def\uj{\mathcal{U}_J}
\def\lie{ \mathfrak{u}_J}

\newcommand{\ra}{\rightarrow}
\newcommand{\ol}{\overline}
\newcommand{\vareps}{\varepsilon}
\newcommand{\x}{\times}
\newcommand{\ort}{[\bot]}

\newcommand{\Q}{\mathcal{Q}}
\newcommand{\M}{\mathcal{M}}
\newcommand{\sdo}{[\dotplus]}

\newcommand{\PI}[2]{\left\langle #1 , #2 \right\rangle}
\newcommand{\K}[2]{\left [ #1 , #2 \right ]}
\newcommand{\PK}[2]{\left [ #1 , #2 \right ]}

\newcommand{\ka}[1]{\kappa_{#1}}

\newtheorem{teo}{Theorem}[section]
\newtheorem{prop}[teo]{Proposition}
\newtheorem{lem}[teo]{Lemma}
\newtheorem{coro}[teo]{Corollary}

\theoremstyle{definition}
\newtheorem{rem}[teo]{Remark}

\begin{document}

\title{On the geometry of normal projections in Krein spaces} 

\author{E. Chiumiento\footnote{Partially supported by  PIP CONICET 0757 and UNLP 11X585.},\ A. Maestripieri\footnote{Partially supported by  PIP CONICET 0757.}\ and\ F. Mart\'{\i}nez Per\'{\i}a\footnote{Partially supported by PIP CONICET 0435 and UNLP 11X585.}}

\date{}

\maketitle

\begin{abstract}
Let $\h$ be a Krein space with fundamental symmetry $J$. 
Along this paper, the geometric structure of the set of $J$-normal projections $\q$ is studied. The group of $J$-unitary operators $\uj$ naturally acts on $\q$. Each orbit of this action turns out to be an analytic homogeneous space of $\uj$, and a connected component of $\q$. 

The relationship between $\q$ and the set $\e$ of $J$-selfadjoint projections is analized: both sets are analytic submanifolds of $L(\h)$ and there is a natural real analytic submersion from $\q$ onto $\e$, namely $Q\mapsto QQ^\#$.

The range of a $J$-normal projection is always a pseudo-regular subspace. Then, for a fixed pseudo-regular subspace $\s$, it is proved that the set of $J$-normal projections onto $\s$ is a covering space of the subset of $J$-normal projections onto $\s$ with fixed regular part. 
\end{abstract}



\maketitle

\section{Introduction}

Let $\h$ be a Krein space with fundamental symmetry $J$. A pseudo-regular subspace of $\h$ is a subspace $\s$ such that $\s=\mathcal{M}[\dot{+}]\s^{\circ}$, where $\mathcal{M}$ is a regular subspace  of $\h$ and $\s^{\circ}$ is the isotropic part of $\s$. For instance, subspaces of Pontryagin spaces are always pseudo-regular. Pseudo-regularity appeared as a condition to generalize
results on spectral measures of definitizable operators  from Pontryagin spaces to general Krein spaces \cite{jonas, GJ}. It was also useful to 
extend the Beurling-Lax theorem for shifts acting on indefinite metric spaces \cite{ball, G98}. 

This paper is devoted to investigating the geometric structure of the set of  $J$-normal projections, namely
\[  \q=\{ \, Q \in L(\h) \, : \, Q^2=Q, \, QQ^{\#}=Q^{\#}Q \, \}, \]
where $L(\h)$ is the algebra of  bounded linear operators acting on $\h$ and $Q^{\#}$ stands for the $J$-adjoint of $Q$. This class of projections is intimately related to the family of closed pseudo-regular subspaces of $\h$. 
In fact, a (closed) subspace $\s$ is pseudo-regular if and only if $\s$ is the range of a $J$-normal projection. However, the correspondence between pseudo-regular subspaces and $J$-normal projections is not bijective: there can be infinitely many $J$-normal projections onto the same subspace, see \cite{maestripieri martinez13}.

An operator  $U \in L(\h)$ is $J$-unitary if $UU^{\#}=U^{\#}U=I$. The set $\uj$ of all $J$-unitary operators is a Banach-Lie group endowed with the norm topology of $L(\h)$. It naturally acts by conjugation on the set of $J$-normal projections, i.e.  if $U \in \uj$ and $Q \in \q$ the action of $U$ on $Q$ is defined by  $U \cdot Q=UQU^{\#}$. 

In this paper, it is shown that, for each $Q_0 \in \q$, the orbit $\uj \cdot Q_0$ is an analytic homogenous space of $\uj$. Thus, each orbit can be endowed with the quotient topology. On the other hand, $\uj \cdot Q_0$ has also the topology inherited from $L(\h)$. But it is shown that both topologies coincide. In order to obtain this result, it is proved that the map induced by the action 
\begin{equation}\label{mapa action}
p_{Q_0} : \uj \to \uj \cdot Q_0, \, \, \, \, \, p_{Q_0}(U)=UQ_0U^{\#}, 
\end{equation}	
has a local continuous cross section (Theorem \ref{section}). In \cite{G}, A. Gheondea found several conditions equivalent to the existence of a $J$-unitary implementing equivalence between two pseudo-regular subspaces. The above mentioned local continuous cross section allows to find a $J$-unitary that (locally) depends continuously on the $J$-normal projections and implements the equivalence between their ranges. As a  consequence, it follows  that the action of $\uj$ on $\q$ fills connected components. Furthermore, the orbits can be characterized by means of the signatures and cosignatures associated to the range of any of its projections (Proposition \ref{comp q}).

The problem of finding a local continuous cross section for the action is central to develop the differential geometry of infinite dimensional smooth homogeneous spaces which arise in operator theory. Several examples can be found in \cite{beltita}.  However, each example  usually requires ad-hoc techniques. In particular, the existence of a section for the map given in \eqref{mapa action} relies on two facts.  First, the section given in \cite{corach por re93} for the set of projections in $L(\h)$. Second, after noticing that the isotropic subspaces of the range and nullspace of a $J$-normal projection are closed neutral companions \cite{GJ}, the construction of biorthogonal bases of the sum of these subspaces for each projection in the orbit.

Concerning the smooth structure of $\q$, it turns out that  $\q$ is an analytic submanifold of $L(\h)$. In particular, the same result holds for the set of $J$-selfadjoint projections 
$$\e=\{ \, E \in L(\h)   \, : \,  E^2=E, \, E^{\#}=E\, \}.$$
These facts allow to understand the relationship between $J$-normal projections and $J$-selfadjoint projections from a geometrical point of view: the map $F:\q \to \e$ defined by $F(Q)=QQ^\#$ is a real analytic submersion (Theorem \ref{q submanifold lh}). This kind of results can be seen as a contribution to the differential geometry of projections, which has been a subject of study in different settings, see e.g. \cite{corach por re90, corach por re93, pr87,  BR07b, andr lar, andruchow ch di 12}.

The last part of this paper deals with a topological description of the set of $J$-normal projections with a prescribed range.  
For a fixed pseudo-regular subspace $\s$ of $\h$, denote by $\q_{\s}$ the set of $J$-normal projections with range $\s$, that is,
$$  
\q_{\s} = \{ \, Q \in \q  \, : \, R(Q)=\s\, \}.  
$$
Unless $\s$ is regular, $ \q_{\s}$ has infinitely many elements. If the isotropic part  $\s^\circ$ is non trivial, each complement $\mathcal{M}$ in the decomposition $\s=\mathcal{M}[\dot{+}]\s^{\circ}$ is regular.  Thus, $\q_{\s}$ can be decomposed as the disjoint union of the decks  
\[   \q_{\s , \mathcal{M}}=\{  \,Q \in \q_{\s} \,  : \, R(QQ^{\#})=\mathcal{M}  \,\}, \]
where $\mathcal{M}$ is any (regular) complement of $\s^{\circ}$ in $\s$. The group  $\u_{\s}$ of all $J$-unitary operators  
leaving $\s$ invariant, acts transitively on $\q_{\s}$ by conjugation. Moreover, the action has the following remarkable property: the restriction to $\u_{\s}$ of the map defined in \eqref{mapa action} admits a global continuous cross section (Proposition \ref{sec global qs}). This is the key to prove that  $\q_{\s}$ is a covering space of any of the decks $\q_{\s , \mathcal{M}}$ (Theorem \ref{cov spa fixed range}).  

The contents of this paper are as follows. Section 2 contains notation and preliminaries on Krein spaces. Section 3 has the construction of the continuous local cross section for the natural action of $\uj$ on $\q$. The differential structure of $\q$  is developed in Section 4. Finally, Section 5 presents the covering space structure of the $J$-normal projections with a prescribed range. 

\section{Preliminaries}

Let $(\h, \PI{\,}{\,})$ be a complex separable Hilbert space. If $\k$ is another Hilbert space, $L(\h,\k)$ stands for the vector space of bounded linear operators from $\h$ to $\k$. In particular, $L(\h)$ is the algebra of bounded operators  on $\h$. 
If $T \in L(\h)$, $T^*$ is the adjoint of $T$. The range  and the nullspace of $T$ are denoted by  $R(T)$ and $N(T)$, respectively. 
The spectrum of $T$ is denoted by $\sigma(T)$. 

Throughout this paper, $J$ is a fixed symmetry acting on $\h$ (i.e. $J=J^*=J^{-1}$), which defines a
fundamental decomposition $\h=\h_+ \oplus \h_-$ given by $\h_\pm=N(J\mp I)$. This symmetry induces a Krein space structure $(\h,\PK{\,}{\,})$, where
\[  \PK{f}{g}=\PI{Jf}{g}, \, \, \, \, \, \, \, f,g \in \h.   \] 
The orthogonal projection onto $\h_{\pm}$ is denoted by $P_{\pm}$.
For a detailed exposition of the facts below, and a deeper discussion on Krein spaces see \cite{ando79, azizov iok89, bognar}.

A vector $f \in \h$ is $J$-positive if $\PK{f}{f}>0$. A subspace $\s$ of $\h$ is $J$-positive if every nonzero vector $f \in \s$ is $J$-positive.  A subspace $\s$ is called uniformly $J$-positive if there is  a constant $c>0$ such that $\PK{f}{f}\geq c \|f\|^2$ for every $f \in \s$. 
A $J$-positive (resp. uniformly $J$-positive) subspace is said to be maximal if it is not properly contained in a larger $J$-positive (resp. uniformly $J$-positive) subspace.  Similarly, one can define $J$-nonnegative, $J$-neutral, $J$-negative and  uniformly $J$-negative subspaces.

For each $J$-positive subspace $\s$ of $\h$, the angular operator $K: P_{+} (\s) \to \h_-$ is defined by $K(P_+ f)=P_- f$. It is a contraction ($\|K\|\leq 1$) and its graph coincides with $\s$: 
\[
Gr(K)\simeq \{ \, f + Kf \, : \, f \in P_{+} (\s) \, \}=\s.
\] 
Moreover, $K$ is a uniform contraction ($\|K\|< 1$) if and only if $\s$ is uniformly $J$-positive. If $\s$ is maximal (in the corresponding class of subspaces) the $P_+ (\s)=\h_+$. Observe that the angular operator can also be defined for $J$-negative subspaces in the obvious way. 

 Let $\s$ be a subspace of $\h$. The $J$-orthogonal subspace of $\s$ in $\h$ is defined by 
\[   \s^{[\perp]}=\{ \, f \in \h \, : \,  \PK{f}{g}=0 \ \text{for every  $g \in \s$}  \,  \}. \]
The isotropic part of $\s$ is given by $\s^{\circ}:=\s \cap \s^{[\perp]}$. In general, it is a non-trivial subspace. A subspace $\s$  is $J$-non-degenerate if $\s \cap \s^{[\perp]}=\{ 0\}$. Otherwise, it is a $J$-degenerated subspace. If $\mathcal{T}$ is another subspace of $\h$, $\s\dot{+}\mathcal{T}$ stands for the direct sum of the subspaces, meanwhile $\s[\dot{+}]\mathcal{T}$ is the $J$-orthogonal (direct) sum of them.

Given $T \in L(\h)$, the $J$-adjoint operator of $T$ is defined by $T^{\#}=JT^*J$. An operator $T$ is $J$-selfadjoint if $T^{\#}=T$, and it is $J$-antihermitian if $T^{\#}=-T$.

\subsection{The $J$-unitary group}

A $J$-unitary operator $U$ is a surjective isometry respect to the indefinite inner product, i.e. an operator satisfying $\K{Uf}{Uf}=\K{f}{f}$ for every $f\in\h$. Observe that it is possible to find unbounded $J$-unitary operators in Krein spaces, see e.g. \cite{G88} and the references therein. Along this work, only bounded $J$-unitary operators are considered. Then, $U \in L(\h)$ is $J$-unitary  if and only if $UU^{\#}=U^{\#}U=I$. The group of all (bounded) $J$-unitary operators is denoted by $\uj$.

\begin{rem}\label{B L subg} Let $Gl(\h)$ denote the group of invertible operators. 
The group of bounded  $J$-unitary operators can be rewritten as
\begin{equation}\label{subg simetrias}
  \uj = \{ \, U \in Gl(\h)  \, : \, U^*JU=J \, \}.   
\end{equation}  
It was mentioned in \cite[Section 23]{upmeier85} that this set is a real Banach-Lie subgroup of 
$Gl(\h)$. In fact $\uj$ turns out to be a real algebraic subgroup of $Gl(\h)$ and, by \cite[Theorem  7.14]{upmeier85}, $\uj$ is a Banach-Lie group endowed with the operator norm topology.

Its Lie algebra $\mathfrak{u}_J$ can be identified with the subspace of $J$-antihermitian operators, i.e.
\[  \mathfrak{u}_J=\{ \, X \in L(\h)  \, : \, X^\#=-X \,  \}.  \]
\end{rem}

When  the Hilbert space $\h$ is considered over a general field,  subgroups of $Gl(\h)$ defined as in (\ref{subg simetrias}) are not necessarily connected. However, if $\h$ is a complex Hilbert space,  $\uj$ is connected. This fact seems to be well-known, but no references could be found by the authors. A proof is included below based on the following well-known description of $J$-unitaries acting on Krein spaces, see e.g. \cite{ando79}.  

\begin{teo}\label{ando caracteriz}
Let $\s$ be a maximal uniformly $J$-positive subspace with angular operator $K$. Then, for any choice of unitary operators $V_+$ on $\h_+$ and $V_-$ on $\h_-$ the block operator matrix $U$ with respect to the decomposition $\h=\h_+\oplus\h_-$ given by
\[ U = \left( \begin{array}{ll} (I_+ - K^*K)^{-1/2} V_+ & K^*(I_- - KK^*)^{-1/2}V_- \\ K(I_+ - K^*K)^{-1/2} V_+ & (I_- - KK^*)^{-1/2}V_-  \end{array} \right)\]
is $J$-unitary and transforms $\h_+$ onto $\s$. Conversely, every $J$-unitary operator that maps $\h_+$
onto $\s$ is of this form.
\end{teo}

\begin{prop}\label{connected uj}
The Banach-Lie group $\uj$ is (arcwise) connected.
\end{prop}
\begin{proof}
Let $U$ be a  $J$-unitary operator. It is not difficult to see that $U\h_+$ is  a maximal uniformly $J$-positive subspace. By Theorem \ref{ando caracteriz}, $U$ can be written in the form 
\[ U = \left( \begin{array}{ll} (I_+ - K^*K)^{-1/2} V_+ & K^*(I_- - KK^*)^{-1/2}V_- \\ K(I_+ - K^*K)^{-1/2} V_+ & (I_- - KK^*)^{-1/2}V_-  \end{array} \right), \]
where $K$ is the angular operator of $U\h_+$. Here $V_+$ and $V_-$ are unitary operators on $\h_+$ and $\h_-$ respectively. Then,  there exist antihermitian operators $X_+$  acting on $\h_+$ and $X_-$ acting on $\h_-$ such that $V_{+}=e^{X_+}$ and $V_-=e^{X_-}$. Notice that the operators of the form $e^{tX_{\pm}}$ are unitaries on $\h_{\pm}$ for  $t \in \RR$, and $tK$ is a uniform contraction for  $t \in [0,1]$. For each $t \in [0,1]$, the uniform contraction $tK$ is uniquely associated to a maximal uniformly $J$-positive subspace, see \cite[Corollary 1.1.2]{ando79}. 
Therefore, by Theorem \ref{ando caracteriz}, the curve $\gamma: [0,1]\ra L(\h)$ given by
\[   \gamma(t)= \left( \begin{array}{ll} (I_+ - t^2K^*K)^{-1/2} e^{tX_+} & tK^*(I_- - t^2KK^*)^{-1/2}e^{tX_-} \\ tK(I_+ - t^2K^*K)^{-1/2} e^{tX_+} & (I_- - t^2KK^*)^{-1/2}e^{tX_-}  \end{array} \right), \]
takes values on $\uj$. Moreover, this curve is clearly continuous, and it joins $\gamma(0)=I$ with $\gamma(1)=U$. Thus, every $J$-unitary operator  can be joined by means of  a continuous curve  with the identity. 
Hence $\uj$ is arcwise connected.   
\end{proof}

\begin{rem}\label{log cerca de 1} 
The exponential map $\exp: \mathfrak{u}_J \to \uj$ is given by $\exp(X)=e^X$. It is always a local diffeomorphism. A surjectivity radius of the exponential map can be estimated as follows: 

Let $U \in \uj$ such that $\|U-I\|<1$.  Consider the principal branch of the logarithm, i.e. $\log: \CC\setminus \RR^-\ra \CC$ given by $\log(z)=\log(|z|) + i\theta$, where $z=|z|e^{i\theta}$, $\theta\in (-\pi,\pi)$. Since every $\lambda \in  \sigma(U)$ satisfies $|\lambda - 1|<1$, the logarithm of $U$ can be defined by the analytic functional calculus:
\[
\log(U)=\frac{1}{2\pi i}\int_\Gamma \log(z)(zI-U)^{-1}\ dz,
\]
where $\Gamma$ is a suitable Jordan contour in the resolvent set of $U$ surrounding $\sigma(U)$. 

Observe that $\sigma((U^*)^{-1})=\{\ol{\lambda^{-1}}: \ \lambda\in\sigma(U)\}$ is also contained in the right-half plane. Then there exist $0<\vareps<M$ and $N>0$ such that $\sigma(U)\cup\sigma((U^*)^{-1})$ is contained in the rectangle $[\vareps,M]\x [-N, N]$. Let $\Gamma$ be the border of this rectangle. 
Since $U\in\uj$, it follows that $JU=(U^*)^{-1}J$ and, given $z\in\CC$, 
$
(zI-U)^{-1}J=(J(zI-U))^{-1}=((zI-(U^*)^{-1})J)^{-1}=J(zI-(U^*)^{-1})^{-1}.
$
Thus, 
\begin{eqnarray*}
\log(U)J &=& \frac{1}{2\pi i}\int_\Gamma \log(z)(zI-U)^{-1}J\ dz \\
&=& J\left(\frac{1}{2\pi i}\int_\Gamma \log(z)(zI-(U^*)^{-1})^{-1} dz\right) =J \log((U^*)^{-1}).
\end{eqnarray*}
Note that $f(z):=\log(\tfrac{1}{z})$ is an analytic function in the right-half plane satisfying $f(z)=-log(z)$. Then, $\log(U)J= J\log((U^*)^{-1})=-J\log(U^*)=-J\log(U)^*$.
Set $X=\log(U)$. By the above computation $X$ is $J$-antihermitian and $e^X=U$.  Hence, every operator $U$ satisfying $\|U-I\|< 1$ has a logarithm in $\mathfrak{u}_J$.
\end{rem}

\subsection{Regular and pseudo-regular subspaces} 

A (closed) subspace $\s$ of a Krein space $\h$  is called \textit{regular} if $\s \, [\dot{+}]\, \s^{[\perp]}=\h$. Equivalently, $\s$ is regular if and only if there exists a (unique) $J$-selfadjoint projection $E$ such that $R(E)=\s$ (see e.g. \cite[Ch. 1, Thm. 7.16]{azizov iok89}). Thus, the set of regular subspaces is in bijective correspondence with the set of $J$-selfadjoint projections, namely, 
\[  \e=\{ E \in L(\h) \, : \,  E^2=E, \, E^{\#}=E\,   \}. \]
The following criterion will be useful: $\s$ is a regular subspace if and only if $\s=\m \, [\dot{+}] \, \n$, where 
$\m$ is a uniformly $J$-positive subspace and $\n$ is a uniformly $J$-negative subspace (see \cite[Theorem 1.3]{ando79}). 

A closed subspace $\s$ of $\h$  is called \textit{pseudo-regular} if there exists a regular subspace $\m$ such that $\s=\s^\circ \, [\dot{+}]\, \m$.  Equivalently, $\s$ is pseudo-regular  if the algebraic sum $\s \, +\, \s^{[\perp]}$ is closed. In \cite{maestripieri martinez13} it was shown that a subspace $\s$ is pseudo-regular if and only if $\s$ is the range of a $J$-normal projection, i.e. there exists a projection $Q\in L(\h)$ with $R(Q)=\s$ such that $QQ^\#=Q^\#Q$. 

The following results also belong to  \cite{maestripieri martinez13}. Their statements are included in order to make the paper self-contained.

\begin{prop}\label{desc}
Given a projection $Q \in L(\h)$, $Q$ is $J$-normal if and only if there exist a projection $E \in \e$ and a projection $P \in L(\h)$ satisfying $PP^{\#}=P^{\#}P=0$ such that 
\[ Q=E+P.\]
The projections $E$ and $P$ are uniquely determined by $Q$. 
\end{prop}

Projections $P\in L(\h)$ satisfying $PP^{\#}=P^{\#}P=0$ were previously considered in \cite{J81,GJ}, in connection with neutral dual companions. If $\s$ is a fixed (closed) $J$-neutral subspace of $\h$, a {\it neutral dual companion of $\s$} is another (closed) $J$-neutral subspace $\t$ of $\h$ such that $\h=\s  \dotplus \t^{\ort}$ holds. If $\t$ is a neutral dual companion of $\s$ then also $\h=\t  \dotplus \s^{\ort}$ holds. So, the pair of subspaces $(\s,\t)$ is called a {\it neutral dual pair}.

\begin{rem}\label{desc q y 1-q}
In the proof of the above mentioned result, the projections $E$ and $P$ are explicitly computed: $E=QQ^{\#}$ and $P=Q(I-Q^{\#})$. 
Furthermore, the decomposition for the $J$-normal projection $I-Q$ is given by 
\[ I-Q=F + P^{\#}, \]
where $F=(I-Q)(I-Q)^{\#}$. From these formulas, it is easy to see that $EP=PE=EP^\#=P^\#E=0$ and $FP=PF=FP^\#=P^\#F=0$. 

Also, it follows that  $R(P + P^\#)=R(Q)^{\circ} \dot{+} N(Q)^{\circ}$ is a regular subspace, and the Krein space $\h$ can be decomposed as the $J$-orthogonal sum of the following three
regular subspaces:
$$
 \h=R(E)\sdo R(P + P^{\#})\sdo R(F).
$$
\end{rem}

In the sequel, given a  a $J$-normal projection $Q\in L(\h)$, $E$, $F$ and $P$ stand for the projections $E=QQ^{\#}$, $P=Q(I-Q^{\#})$ and $F=(I-Q)(I-Q)^{\#}$. If $Q_0$ is another $J$-normal projection, $E_0$, $F_0$ and $P_0$ have the obvious meaning.

\section{The orbit of a $J$-normal projection}

\noi The set of $J$-normal projections is given by
\[  \q=\{ \, Q \in L(\h) \, : \, Q^2=Q, \, QQ^{\#}=Q^{\#}Q \, \}. \]
The Banach-Lie group $\uj$ acts smoothly on $L(\h)$ by conjugation. Clearly, the restriction of this action gives an action of  $\uj$ on $\q$ defined by 
$$U\cdot Q=UQU^{\#},$$ 
where $U \in \uj , \, \, Q \in \q$. It is worth pointing out that each orbit $\uj \cdot Q$ is connected in the norm topology (see Proposition \ref{connected uj}).
For the notion of real analytic homogeneous space in the following result see e.g. \cite{beltita, upmeier85}.

\begin{prop}\label{kernel sup}
Given $Q_0 \in \q$, the orbit $\uj \cdot Q_0$ is a real analytic homogeneous space of $\uj$.
\end{prop}
\begin{proof}
Clearly, there is a bijection from $\uj \cdot Q_0$ onto $\uj / \mathcal{G}$, where $\mathcal{G}$ is the isotropy group at $Q_0$, i.e. 
\[ \mathcal{G}= \{ \, U \in \uj \, : \, UQ_0=Q_0U \, \}.  \]
Observe that the Lie algebra of $\mathcal{G}$ can be identified with
\[  \mathfrak{g}=  \{ \, X \in \lie \, : \, XQ_0=Q_0X \, \}.   \]
Then, the conclusion of this proposition will follow if $\mathcal{G}$ is a Banach-Lie subgroup of $\uj$.
This last fact is equivalent to show that $\mathcal{G}$ is a Banach-Lie group in the norm topology of $L(\h)$ and $\mathfrak{g}$ is a closed complemented subspace of $\lie$. In this case, $\uj / \mathcal{G}$ has an analytic manifold structure endowed with the quotient topology (see e.g. \cite[Theorem 8.19]{upmeier85}). 

Let $\mathcal{V}=\exp^{-1}(B_1 (I))$, where $B_1 (I)$ is the open unit ball around  the identity contained in $\uj$. Given $U \in \exp(\mathcal{V}) \cap \mathcal{G}$, there exists $X \in \mathcal{V}$ such that $U=e^X$.  

Notice that the logarithm $X \in \lie$, which is computed in Remark \ref{log cerca de 1}, is unique. Indeed, 
if $\| U -I\|<1$ then $\sigma(U)\subset \RR + i (-\pi, \pi)$. But in the latter set the complex exponential is bijective, so the exponential in $L(\h)$ is also bijective  by well-known properties of the functional analytic calculus. Now recall that $X=\frac{1}{2\pi i}\int_\Gamma \log(z)(zI-U)^{-1}dz$. If the operator $U$ belongs to $\mathcal{G}$, that is $UQ_0=Q_0U$, then by standard arguments one can see that $XQ_0=Q_0X$. Thus, $X \in \mathfrak{g}$. This shows that $\exp(\mathcal{V})\cap \mathcal{G}\subseteq \exp(\mathcal{V} \cap \mathfrak{g})$. Since the reversed inclusion is always trivial, it follows that 
$\exp(\mathcal{V})\cap \mathcal{G}=\exp(\mathcal{V} \cap \mathfrak{g})$. Hence $\mathcal{G}$ is a Banach-Lie group in the norm topology of $L(\h)$.

Note that $\mathfrak{g}$ is closed in $\lie$. To prove that  $\mathfrak{g}$ is complemented in $\lie$, consider the map 
\begin{equation}\label{Pmap}
\p: L(\h)\ra L(\h), \, \, \, \, \p(X)=E_0XE_0 + P_0XP_0 + P_0^{\#}XP_0^{\#} + F_0XF_0.
\end{equation}
 By the relations between the projections $E_0$, $F_0$, $P_0$ and $P_0^{\#}$ pointed out in Remark \ref{desc q y 1-q},  it follows that $\p$ is a  continuous projection satisfying $\p(\mathfrak{u}_J)\subseteq \mathfrak{u}_J$. 
Also, notice that 
$
Q_0\p(X)=E_0XE_0 + P_0XP_0=\p(X)Q_0$.
 Then, $\p(\mathfrak{u}_J)\subseteq \mathfrak{g}$. To prove the reversed inclusion, pick $X\in \mathfrak{g}$, i.e. $X\in\mathfrak{u}_J$ and $XQ_0=Q_0X$. Observe that $X$ also commutes with $Q_0^{\#}$. Therefore, $X$ commutes with $E_0$, $F_0$, $P_0$ and $P_0 ^{\#}$, so that
\[
\p(X)=E_0XE_0 + P_0XP_0 + P_0^\#XP_0^\# + F_0XF_0= (E_0 + P_0 + P_0^\# + F_0)X=X.
\]
The latter means that $X\in \p(\mathfrak{u}_J)$, and consequently, $\p(\mathfrak{u}_J)= \mathfrak{g}$. 

To finish the proof, note that the map $h :L(\h) \to \mathfrak{u}_J$ given by $h(X)=\tfrac{X-X^\#}{2}$ is a continuous real projection. Therefore the map 
$\p\circ h$ is a continuous real projection onto $\mathfrak{g}$. Hence, $\mathfrak{g}$ is complemented in $L(\h)$.
\end{proof}

According to the above proposition, the orbit $\uj \cdot Q_0$ has a Banach manifold structure such that 
the canonical projection 
$$p_{Q_0}:\uj \to \uj \cdot Q_0, \, \, \, \, p_{Q_0}(U)=UQ_0U^{\#}$$ 
is a real analytic submersion. This manifold structure defines on  $ \uj \cdot Q_0 \simeq \uj / \mathcal{G}$  the quotient topology. On the other hand, $\uj \cdot Q_0$ is endowed with the relative topology as a subset of $L(\h)$. If one considers the identity map $Id: \uj \cdot Q_0 \simeq \uj /\mathcal{G} \to  \uj \cdot Q_0 \subseteq L(\h) $, it is easy to see that this map is always continuous. However,  it may fail to be a homeomorphism. 
To see that in this setting it is actually a homeomorphism, it will be sufficient to prove that
 $p_{Q_0}$ admits local continuous cross sections when $\uj \cdot Q_0$ is endowed with the relative topology of $L(\h)$. 

To this end, recall that in Remark \ref{desc q y 1-q}, it is stated that $\h$ can be written as the $J$-orthogonal sum of three regular subspaces
\[   \h=R(E_0)[\dot{+}]R(P_0 + P_0^{\#})[\dot{+}]R(F_0). \]
Let $Q$ be another $J$-normal projection sufficiently close to $Q_0$. The space $\h$ can   also be decomposed as $\h=R(E)[\dot{+}]R(P + P^{\#})[\dot{+}]R(F)$. Therefore, the problem of finding a $J$-unitary that maps $Q_0$ in $Q$ can be reduced to find $J$-isometric isomorphisms mapping $R(E_0)$ onto $R(E)$, $R(F_0)$ onto $R(F)$ and $R(P_0+P_0^{\#})$ onto $R(P+P^{\#})$. It is worth pointing out that $R(P_0)$ has to be mapped onto $R(P)$, and obviously, the $J$-unitary has to depend continuously on $Q$. 

This work is carried out in the next two subsections. The first one deals with the case of $J$-selfadjoint projections, and the second subsection treats the general case.

\subsection{A local continuous cross section. The $J$-selfadjoint case.}

Observe that the group $\uj$ also acts on $\e$ by conjugation: $U \cdot E=UEU^{\#}$, where $U \in \uj$ and $E \in \e$. 

\begin{prop}\label{section j seladj}
The map $p_{E_0}:\uj \to \uj \cdot \e$ given by $p_{E_0}(U)=UE_0U^{\#}$ has local continuous cross sections.
\end{prop}
\begin{proof} In what follows,  a section will be given in a neighborhood of $E_0$; standard arguments can be applied to translate this section to other points.

 It will be useful to recall some facts on the geometry of  projections in $L(\h)$, see \cite{corach por re93}.
The set of projections in $L(\h)$, namely
\[  \mathbb{Q}=\{  \, Q \in L(\h) \, : \, Q^2=Q \, \},   \] 
is a smooth homogeneous space of the group $Gl(\h)$.  Its tangent space at $Q \in {\mathcal Q}$ can be identified with
\[  T_{Q} \mathbb{Q}=\{  \, X \in L(\h) \, :\,  XQ+QX=X \,  \},  \] 
which are co-diagonal operators with respect to $Q$, i.e. co-diagonal block-operator matrices according to the decomposition $\h=R(Q)\dotplus N(Q)$. For a fixed projection $Q_0\in \mathbb{Q}$,  the exponential map 
\[ \exp: T_{Q_0}\mathbb{Q} \to  \{ GQ_0G^{-1}: \ G \in Gl(\h)\}, \, \, \, \, \,  \exp(X)=e^XQ_0e^{-X},   \]
is a local diffeomorphism at $Q_0$. Therefore, there is a positive radius $r$ (depending on $Q_0$) such that the map $\{  \, Q \in \mathbb{Q}  \, : \, \| Q - Q_0\|<r \, \} \to    T_{Q_0}\mathbb{Q}$ given by $Q \mapsto X_Q$ is smooth and satisfies 
\[
e^{X_Q}Q_0e^{-X_Q}=Q.
\]  
Taking into account the facts stated above for the projection $E_0 \in \e$, given a suitable radius $r$, it is possible to define a continuous map
\begin{equation}\label{radio}  
s: \{ \, E \in \e  \, : \, \|E-E_0\|<r \, \} \to Gl(\h), \, \, \, \, \,  s(E)=e^{X_E}. 
\end{equation}
If this map takes values in $\uj$, it will clearly be the required continuous local cross section for $p_{E_0}$. The following argument to show that $s(E) \in \uj$
is borrowed and adapted from \cite[Proposition 4.4]{andruchow ch di 12}. It is useful to change from projections to symmetries via the map $E \mapsto R_E=2E-I$. Since $e^{X_E}E_0e^{-X_E}=E$, it follows that $e^{X_E}R_{E_0}e^{-X_E}=R_E$. Next, notice that an operator is co-diagonal with respect to $E_0$ if and only if it anticommutes with $R_{E_0}$. This implies that $R_{E_0}e^{-2X_E}=R_E=e^{2X_E}R_{E_0}$ and
\[ 
(e^{2X_E})^{\#}=(R_ER_{E_0})^{\#}=R_{E_0} R_E=e^{-2X_E}=(e^{2X_E})^{-1}.       
\]
Then, $e^{2X_E} \in \uj$. Shrinking the radius $r$ if it is necessary, one gets that $\| e^{2X_E} - I \|<1$. By Remark \ref{log cerca de 1}, it follows that $2X_E \in \mathfrak{u}_J$, and consequently, $X_E \in \mathfrak{u}_J$. Hence $e^{X_E} \in \uj$ and the proof is completed.
\end{proof}

\subsection{A local continuous cross section. The general case.}

Given a neutral dual pair $(\s, \t)$ in $\h$, in the next lemma a pair of biorthogonal bases for $\s$ and $\t$ are constructed. This result was known for finite-dimensional subspaces \cite[Lemma 1.31]{azizov iok89}, but it is original for the general case.

\begin{lem}\label{biort}
If $(\s, \t)$ is a a neutral dual pair in $\h$, then for any orthonormal basis $\{s_n\}_{n\geq 1}$ of $\s$ (in the Hilbert space sense) there exists a Riesz basis $\{t_n\}_{n\geq 1}$ of $\t$ such that 
\[
\K{s_i}{t_j}=\delta_{ij}, \ \ \ i,j\geq 1.
\]
\end{lem}

\begin{proof}
Let $P\in L(\h)$ be the projection onto $\s$ along $\t^{\ort}$. Then, $P^{\#}$ is the projection onto $\t$ along $\s^{\ort}$. For a fixed orthonormal 
basis $\{s_n\}_{n\geq 1}$ of $\s$, define $t_n=P^{\#}Js_n \in \t$, $n \geq 1$. Hence, given $i,j\geq 1$,
\[
\K{s_i}{t_j}=\K{s_i}{P^{\#}Js_j}=\K{P s_i}{Js_j}=\K{s_i}{Js_j}=\PI{s_i}{s_j}=\delta_{ij}.
\]
To prove that $\{t_n\}_{n \geq 1}$ is a Riesz basis, observe that $T=P^{\#}J|_{\s}:\s \to \t$ is a (continuous) surjective operator since
\[ 
T(\s)=P^{\#}( J(\s) + N(P^{\#}))= R(P^{\#})=\t.
\]  
On the other hand, if $f\in N(T)$ then $Jf \in \s^{\ort}=J(\s^{\perp})$. It follows that 
$f \in \s^\bot\cap \s$. Thus, $T$ is injective. Hence, $\{t_n\}_{n \geq 1}$ is the image of an orthonormal basis by an invertible operator, i.e. it is a Riesz basis.
\end{proof}

If $Q$ is a $J$-normal projection, notice that the subspaces $R(Q)^\circ$ and $N(Q)^\circ$ form a neutral dual pair.
\begin{lem}\label{s link}
Let $Q, Q_0\in L(\h)$ be $J$-normal projections. Assume that the isotropic parts of their ranges have the same dimension. 
Then, there is a continuous $J$-isometric isomorphism 
\[
V: R(Q_0)^\circ  \dot{+} N(Q_0)^\circ \to R(Q)^\circ  \dot{+} N(Q)^\circ,
\]
satisfying $V(R(Q_0)^\circ)=R(Q)^\circ$ and $V(N(Q_0)^\circ)=N(Q)^\circ$.
\end{lem}
\begin{proof}
According to Lemma \ref{biort}, for fixed orthonormal basis $\{s_n^0\}_{n\geq 1}$ and $\{s_n\}_{n\geq 1}$ of $R(Q_0)^\circ$ and $R(Q)^\circ$, there exist Riesz basis $\{t_n^0\}_{n\geq 1}$ and $\{t_n\}_{n\geq 1}$ of $N(Q_0)^\circ$ and $N(Q)^\circ$, respectively, such that $\K{s_i^0}{t_j^0}=\K{s_i}{t_j}=\delta_{ij}$.

Next, consider the operator $V:R(Q_0)^\circ  \dot{+} N(Q_0)^\circ \to R(Q)^\circ \dot{+} N(Q)^\circ$ given by
\[
V\bigg(\sum_{n} \alpha_n s_n^0 + \sum_m\beta_m t_m^0 \bigg)=\sum_{n} \alpha_n s_n + \sum_m\beta_m t_m .
\]
Since $V$ maps the (Riesz) basis $\{s_n^0\}_{n\geq 1}\cup \{t_n^0\}_{n\geq 1}$ onto the (Riesz) basis $\{s_n\}_{n\geq 1}\cup \{t_n\}_{n\geq 1}$, it follows that it is a  continuous operator. 

Moreover, $V$ is a $J$-isometry by construction: due to the $J$-biorthogonal\-i\-ty of the bases, it follows that 
\begin{flalign*}
&\K{V\bigg(\sum_{n} \alpha_n s_n^0 + \sum_m\beta_m t_m^0 \bigg)}{V\bigg(\sum_{n} \alpha_n s_n^0 + \sum_m\beta_m t_m^0 \bigg)}  &\\
&= \K{\sum_{n} \alpha_n s_n + \sum_m\beta_m t_m }{\sum_{n} \alpha_n s_n + \sum_m\beta_m t_m }  = 2 \sum_{n,m} \text{Re} (\alpha_n \bar{\beta}_m[s_n,t_m]) & \\
& = 2 \sum_{n} \text{Re} (\alpha_n \bar{\beta}_n) = \K{\sum_{n} \alpha_n s_n^0 + \sum_m\beta_m t_m^0}{\sum_{n} \alpha_n s_n^0 + \sum_m\beta_m t_m^0},&
\end{flalign*} 
where in the second equality, it is taken into account that $\sum_{n} \alpha_n s_n\in R(Q)^\circ$ and $\sum_m\beta_m t_m\in N(Q)^\circ$, and in the last equality, it is used that $\sum_{n} \alpha_n s_n^0\in R(Q_0)^\circ$ and $\sum_m\beta_m t_m^0\in N(Q_0)^\circ$. Hence, $V$ is a $J$-isometric isomorphism.
\end{proof}

 The next step is to show that the above $J$-isometric isomorphism $V$ depends continuously on $Q$. 
Some basic facts on the geometry of the unitary group and the space of selfadjoint projections will be needed. Let $\mathcal{U}$ be the unitary group of $L(\h)$, and $\mathcal{P}$ be the manifold of selfadjoint projections, i.e.
\[  \mathcal{P}=\{ \, P \in L(\h)\, : \,P=P^2 , \, P=P^* \, \}. \]
The natural action of $\mathcal{U}$ on $\mathcal{P}$ given by $U\cdot P=UPU^*$ has local continuous cross sections. Although this fact was pointed out in \cite[Remark 3.2]{corach por re90},  in the following lemma a short proof is included for the sake of completeness. 
The main idea is adapted from a similar context in \cite[Proposition 2.2]{andr lar}.

\begin{lem}\label{section on p}
If $P_0 \in \mathcal{P}$ the map $\mathcal{U} \to \mathcal{P}$, given by $U \mapsto UP_0U^*$, has local continuous cross sections.  
\end{lem}

\begin{proof}
Consider the open set 
\[  \mathcal{V}=\{ \, P \in \mathcal{P} \, : \, \|P - P_0\|<1  \, \}.  \]
For $P \in \mathcal{V}$, set  $S=PP_0 +  (I-P)(I-P_0)$. Then it is well-known that $\| S - I \| < 1$. Thus, $S$ is invertible. 
The unitary part of $S$ given by $U=|S^*|^{-1}S$ is a continuous function of $P$. Notice that $SS^*P=PSS^*$, which implies that 
$|S^*|P=P|S^*|$ and $P|S^*|^{-1}=|S^*|^{-1}P$.  Therefore, $PU=P|S^*|^{-1}S=|S^*|^{-1}PS=|S^*|^{-1}SP_0=UP_0$, i.e. $UP_0U^*=P$. Hence 
$U=U(P)$ is a continuous local cross section. \qedhere
\end{proof}

It will be also useful to state here a well-known result on projections.

\begin{lem}(\cite[Ch. I]{kato76})\label{kato proj}
Let $E_1, E_2 \in L(\h)$ be projections. If $P_{R(E_1)}$ and $P_{R(E_2)}$ are the orthogonal projections onto their ranges, respectively, then 
\[  \| P_{R(E_1)} -  P_{R(E_2)} \| \leq \|  E_1- E_2   \|.  \]
\end{lem}

Now, given a fixed $J$-normal projection $Q_0\in L(\h)$, consider the following neighborhood of $Q_0$:
\[   
\mathcal{V}_{Q_0}=\bigg\{  \, Q \in \mathcal{Q}  \, : \,  \| Q - Q_0 \| < \frac{1}{2(1 + \|Q_0\|)}  \,  \bigg\}. 
\]
Then, define a map $V: \mathcal{V}_{Q_0} \ra L(\h)$ such that $V(Q)$ is a $J$-isometric isomorphism between $R(Q_0)^\circ + N(Q_0)^\circ$ and $R(Q_0)^\circ + N(Q_0)^\circ$ as follows:

Given $Q\in \mathcal{V}_{Q_0}$, it is easy to see that  $\|Q\|<\|Q_0\|+1$.  Recall that $P=Q(I-Q^\#)$ and $P_0=Q_0(I-Q_0^\#)$, then
\begin{equation}\label{q vs p}
 \| P - P_0\| \leq \| Q- Q_0 \| + \|QQ^\# - Q_0Q_0^\# \| \leq 2 (1 + \|Q_0\|)\|Q - Q_0\|<1.     
\end{equation} 
According to Lemma \ref{kato proj}, it follows that
\begin{equation}\label{p vs p ort}
 \| P_{R(Q)^\circ} - P_{R(Q_0)^\circ} \| \leq \| P- P_0\| < 1. 
\end{equation}  
By Lemma \ref{section on p}, there exists a unitary operator $U=U(P_{R(Q)^\circ})$, which depends
continuously on $P_{R(Q)^\circ}$, and satisfies $UP_{R(Q_0)^\circ}U^*=P_{R(Q)^\circ}$. 
In particular, this implies that $\dim R(Q)^\circ=\dim R(Q_0)^\circ$. 

Moreover, for a fixed orthonormal basis $\{s_n^0\}_{n\geq 1}$ of $R(Q_0)^\circ$, this $U\in\mathcal{U}$ gives a procedure to choose an orthonormal basis of $R(Q)^\circ$: set $s_{n,Q}=Us_n ^0$ for every $n\geq 1$. 

According to Lemma \ref{biort}, there are Riesz bases $\{t_{n,Q}\}_{n \geq 1}$ and $\{t_n ^0\}_{n \geq 1}$ of $N(Q)^\circ$ and $N(Q_0)^\circ$, respectively, such that $\K{s_n^0}{t_m^0}=\K{s_{n,Q}}{t_{m,Q}}=\delta_{nm}$. Applying Lemma \ref{s link}, one can construct 
a $J$-isometric isomorphism $ V(Q)$ between $R(Q_0)^\circ  \dot{+} N(Q_0)^\circ$ and $R(Q)^\circ  \dot{+} N(Q)^\circ$. In fact, it will be useful to 
extend this linear operator to $\h$, i.e.
$$ V(Q) f :=\left\{
\begin{array}{cl}
\displaystyle{\sum_{n} \alpha_n s_{n,Q} + \sum_m\beta_m t_{m,Q} }  & \text{ if }  f=\displaystyle{\sum_{n} \alpha_n s_n^0 + \sum_m\beta_m t_m^0};\\ 
\\
0  & \text{ if }  f \in (R(Q_0)^\circ  \dot{+} N(Q_0)^\circ)^{\ort}.
\end{array}\right.
$$

\begin{lem}\label{sec cont lema}
The map  $V:\mathcal{V}_{Q_0} \to L(\h)$ defined above is continuous.
\end{lem}
\begin{proof}
Let $\{ Q_k\}_{k\geq 1}$ be a sequence in $\mathcal{V}_{Q_0}$. Assume that $\|Q_k - Q\|\to 0$ for some $Q \in \mathcal{V}_{Q_0}$. Let $U_k=U(P_{R(Q_k)^\circ})$ be the unitary associated with $P_k=Q_k(I-Q_k^\#)$ defined after \eqref{p vs p ort}.  Analogously, let $U$  and $P$ be the the corresponding unitary and projection associated with $Q$.

Pick a vector $f=\sum_{n} \alpha_n s_n^0 + \sum_m\beta_m t_m^0 \in R(Q_0)^\circ \dot{+}  N(Q_0)^\circ$, where $\{s_n^0\}_{n\geq 1}$ is an orthonormal basis of $R(Q_0)^\circ$ and $\{t_n^0\}_{n\geq 1}$ is the Riesz basis of $N(Q_0)^\circ$ given by Lemma \ref{biort}. 

In order to prove the continuity of the map $V$, note that 
\begin{align*}
\|(V(Q_k) - V(Q) )f \| & = \|\sum_{n} \alpha_n (s_{n, Q_k} - s_{n, Q}) + \sum_m\beta_m (t_{m, Q_k} - t_{m, Q}) \| \\
& = \| \sum_{n} \alpha_n (U_k  - U )s_n^0 + \sum_m\beta_m (P^{\#}_k - P^{\#})Js_m^0 \| \\
& \leq \|U_k - U \|  \, \| \sum_{n} \alpha_n s_n^0 \| + \| P^{\#}_k - P^{\#} \| \,  \| \sum_m \beta_m s_m^0 \| \\
& \leq \|U_k - U \| \, \|f\|  + C \, \| P_k - P \| \, \|f\|, 
\end{align*}
for a suitable $C=c_1 c_2>0$, because 
$$\| \sum_n \beta_n s_n^0 \|=\left(\sum_n |\beta_n|^2\right)^{1/2}\leq c_1 \| \sum_n \beta_n t_n^0 \|\leq c_1 c_2 \|f\|,$$
where $c_1$ is a constant related to the Riesz basicity of $\{t_n^0\}_{n\geq 1}$ and $c_2$ is the norm of the projection onto $N(Q_0)^\circ$ along $R(Q_0)^\circ$.

Therefore,
\[  
\| V(Q_k) - V(Q) \| \leq \|U_k - U \|   + C \, \| P_k - P\|.
\]
From  (\ref{q vs p}) one gets that $\| P_k - P\| \to 0$, so it remains  to show that $ \|U_k - U \|\to 0 $. Lemma \ref{kato proj} implies that  $\|  P_{R(Q_k)^\circ} - P_{R(Q)^\circ}\|\to 0$ and the map given by $P_{R(Q)} \mapsto U(P_{R(Q)})$ is continuous, so that $\| U_k - U\|\to 0$. 
 \end{proof}

Now the main result of this section follows. In particular, when $J=I$, one recovers the connected components of the Grassmann manifold of a Hilbert space, and this topological result can be deduced in a different fashion from the submanifold structure proved in \cite{pr87}.

\begin{teo}\label{section}
Let $Q_0 \in \mathcal{Q}$, then the map  $p_{Q_0}: \uj \to \uj \cdot Q_0$ given by
\[ 
p_{Q_0}(U)=UQ_0U^{\#}, 
\]
has  local continuous cross sections. In particular, the quotient topology and the topology inherited from $L(\h)$ are equivalent in $\uj \cdot Q_0$.
\end{teo}
\begin{proof}
Recall that $E_0=Q_0Q_0^\#$ and $F_0=(I-Q_0)(I-Q_0)^\#$.   By Proposition   \ref{section j seladj}, there is a continuous section $s_1$ of the map $p_{E_0}(U)=UE_0U^\#$. By the same method one can construct  another continuous section $s_2$ of the 
map $p_{F_0}(U)=UF_0U^\#$.  
In fact, these sections are defined in  open balls of radii $r_{E_0}$ and $r_{F_0}$, respectively, contained in $\e$. Set 
$$r_{Q_0}= \min \bigg\{ \frac{r_{E_0}}{1 + 2\|Q_0\|}, \frac{r_{F_0}}{1 + 2\|Q_0\|}, \frac{1}{1 + 2\|Q_0\|}   \bigg\}  .$$
Recall that $V(Q)$ is a continuous function of $Q$ by Lemma \ref{sec cont lema}. Moreover, it satisfies $V(Q)P_0V(Q)^\# =P$. Then the map $s:\{  Q \in \mathcal{Q} : \  \| Q - Q_0 \| < r_{Q_0}  \} \to \uj$ defined by
\[ 
s(Q)=Es_1(E)E_0 + Fs_2(F)F_0 +(P+ P^\#)V(Q)(P_0+ P_0^\#),
\]   
is the required  continuous section for $p_{Q_0}$. To show that $s(Q)\in\uj$ for $Q \in \mathcal{Q}$ with $\| Q - Q_0 \| < r_{Q_0}$, observe that $s(Q)$ can be alternatively written as $s(Q)=s_1(E)E_0 + s_2(F)F_0 + V(Q)(P_0+ P_0^\#)$ or $s(Q)=Es_1(E) + Fs_2(F) +(P+ P^\#)V(Q)$.

Recall that the identity map $Id: \uj \cdot Q_0 \simeq \uj /\mathcal{G} \to  \uj \cdot Q_0 \subseteq L(\h) $ is continuous by definition of the quotient topology. On the other hand, the existence of local continuous cross sections implies that $p_{Q_0}$ is an open map, and consequently, $Id$ is a homeomorphism. This proves the equivalence between both topologies.   
\end{proof}

\subsection{Connected components of $\q$}

It is necessary to recall some terminology used in \cite{G}.  Let $\s$ be a pseudo-regular subspace of $\h$. So, there exists a regular subspace $\m$ such that $\s=\m [\dot{+}] \s^{\circ}$.  Consider a decomposition 
\[
\s= \s^\circ \sdo \, \m_+ \, [\dot{+}] \, \m_-,
\]
where $\m_+$ is a uniformly $J$-positive subspace, $\m_-$ is a uniformly $J$-negative subspace and $\m=\m_+ [\dot{+}] \m_-$. Then, the numbers $\ka{+}(\s)=\dim \m_+$, $\ka{-}(\s)=\dim \m_-$ and $\ka{0}(\s)=\dim \s^\circ$ are called the \emph{positive, negative and isotropic signatures} of $\s$, respectively. It has been shown that these numbers do not depend on the particular decomposition considered (see e.g. \cite[Ch. 1, Thm. 6.7]{azizov iok89}). 
	If $\s$ is pseudo-regular then so is $\s^{\ort}$, and the \emph{positive and negative cosignatures of $\s$} are defined as $c\ka{+}(\s):=\ka{+}(\s^{\ort})$ and $c\ka{-}(\s):=\ka{-}(\s^{\ort})$. 

\begin{prop}(\cite[Proposition 4.6]{G})\label{spatial c}
Let $\s$ and $\t$ be two pseudo-regular subspaces of $\h$. The following statements are equivalent:
\begin{enumerate}
	\item[i)] $\s$ and $\t$ are $J$-unitarily equivalent, i.e. there exists $U\in\uj$ such that $U(\s)=\t$;
	\item[ii)]  $\s$ is $J$-isometrically isomorphic to $\t$ and $\s^{\ort}$ is $J$-isometrically isomorphic to $\t^{\ort}$;
	\item[iii)] $\ka{+}(\s)=\ka{+}(\t)$, $\ka{-}(\s)=\ka{-}(\t)$,  $c\ka{+}(\s)=c\ka{+}(\t)$,  $c\ka{-}(\s)=c\ka{-}(\t)$ 	and 	$\ka{0}(\s)=\ka{0}(\t)$.
\end{enumerate}
 \end{prop}

With the latter result at hand, it is now  straightforward to give a spatial characterization of the orbits.
Moreover, the orbits are the connected components of $\q$.

\begin{prop}\label{comp q}
Let $Q_0,Q \in \q$. The following assertions are equivalent:
\begin{enumerate}
\item[i)] $Q \in \uj \cdot Q_0$.
\item[ii)]  $R(Q)$ and $R(Q_0)$ have the same (three) signatures and the same (two) cosignatures.
\end{enumerate}
Moreover, the connected component of $Q_0$ in $\q$ coincides with  $\uj \cdot Q_0$.
\end{prop}
\begin{proof}
If $Q\in\q$ then $R(Q)=R(Q)^\circ \sdo \m$ and $N(Q)=N(Q)^\circ \sdo \n$, where $\m$ and $\n$ are regular subspaces. Then, it is easy to see that $R(Q)^{\ort}=R(Q)^\circ \sdo \n$ and
\[
c\ka{\pm} (R(Q))= \ka{\pm} (R(Q)^{\ort})=\ka{\pm} (N(Q)).
\]
Hence, the equivalence between $i)$ and $ii)$  follows from applying Proposition \ref{spatial c} to the ranges of two $J$-normal projections.

Let $C_{Q_0}$ be the connected component of $Q_0$.  Recall that $\uj$ is connected (Proposition \ref{connected uj}).  Therefore $\uj \cdot Q_0$ is connected. Hence $\uj \cdot Q_0 \subseteq C_{Q_0}$.  In order to show the converse inclusion, note that the map
\[   Q \mapsto  (\,\ka{+}(R(Q)),c\ka{+}(R(Q)),\ka{-}(R(Q)),c\ka{-}(R(Q)),\ka{0}(R(Q))\,) \]
is continuous. In fact, if $\| Q- Q_0\|<r_{Q_0}$, where $r_{Q_0}$ is defined in the proof of Theorem \ref{section}, then  there is an operator $U \in \uj$ such that $Q=UQ_0U^{\#}$. According to the equivalence $i)$-$ii)$, it follows that 
the five indices must coincide. This proves that the above map is continuous. Since it takes values on a discrete set, the map has to be constant on $C_{Q_0}$. Now if $Q \in C_{Q_0}$, then the five indices associated to $Q$ are equal to those of $Q_0$. Hence there exists a $J$-unitary such that $Q=UQ_0U^{\#}$.
\end{proof}

The connected components of $\e$ can be obtained as a particular case of the above result.

\begin{coro}
Let $E_0,E \in \e$. The following assertions are equivalent:
\begin{enumerate}
\item[i)] $E \in \uj \cdot E_0$.
\item[ii)]  $R(E)$ and $R(E_0)$ have the same (two) signatures and the same (two) cosignatures.
\end{enumerate}
Moreover, the connected component of $E_0$ in $\e$ coincides with  $\uj \cdot E_0$.
\end{coro}

\section{Differential structure of $\q$}

The following is a well-known criterion to determine if a proper subset of a manifold is indeed a submanifold, see \cite[Proposition 8.7]{upmeier85}. 

\begin{prop}\label{sous varietes}
Consider two Banach manifolds $M$ and $N$.  Suppose that $g:M \to N$ is an analytic inmersion and a homeomorphism onto $N'=g(M)$. Then $N'$ is a submanifold of $N$ and the mapping $g:M \to N'$ is bianalytic. 
\end{prop}

This criterion will be  used to show that $\q$ is a submanifold of $L(\h)$. Note that one can restrict to the connected components of $\q$ given by the orbits $\uj \cdot Q_0$, $Q_0 \in \q$. By Proposition \ref{kernel sup}, $\uj \cdot Q_0$ has a manifold structure compatible with the quotient topology. Moreover, the map $p_{Q_0}$ is an analytic submersion with this manifold structure. Equivalently, $p_{Q_0}$ admits local analytic cross sections (see \cite[Corollary 8.3]{upmeier85}). 
Note that the following diagram commutes
\begin{displaymath}
\xymatrix{
\uj \ar[r]^{p_{Q_0}} \ar[dr]_{\tilde{p}_{Q_0}} & \uj \cdot Q_0  \ar[d]^{i}  \hspace{-0.5cm}& \hspace{-.5cm} \simeq \uj /\mathcal{G}  \\
                                   & L(\h)         \hspace{-0.5cm} &  \hspace{-0.5cm}  
}
\end{displaymath}
where $i$ is the inclusion map and  $\tilde{p}_{Q_0}(U)=UQ_0U^{\#}$. The map $\tilde{p}_{Q_0}$ is clearly analytic because it consists in multiplication and inversion in $L(\h)$. The inclusion map can be locally written as $i=\tilde{p}_{Q_0} \circ s$, where $s$ is a analytic section of $p_{Q_0}$. Hence, $i$ is analytic.

 To prove that $i$ is an inmersion (i.e. its differential map is injective and has complemented range), notice that the range of the differential at $Q\in \uj \cdot Q_0$ of $i$ is precisely the tangent space $T_Q (\uj \cdot Q_0)$. The latter is computed as derivatives of smooth curves in the orbit, and it is given by
\[ T_Q (\uj \cdot Q_0)= \{  \, XQ-QX  \, : \, X \in \lie \, \}. \]
On the other hand, it was shown that the quotient and the inherited topologies coincide in the orbits (Theorem \ref{section}). Hence, to see that $\uj \cdot Q_0$ is a submanifold of $L(\h)$ it is sufficient to find a complement of $T_{Q_0} (\uj \cdot Q_0)$ in $L(\h)$.

To this end, if $Q_0\in \Q$ consider again the decompositions 
\[
Q_0=E_0+P_0 \ \ \ \text{and} \ \ \ I-Q_0= F_0 + P_0^\#,
\] 
given in Proposition \ref{desc}. Let  $XQ_0 - Q_0X$ be a tangent vector of the orbit $\uj \cdot Q_0$ at the point $Q_0$. Since $X^\#=-X$ and $E_0=E_0^{\#}$, the $J$-selfadjoint and the $J$-antihermitian parts of $XQ_0-Q_0X$ are given by
$$XE_0 - E_0X  + \frac{1}{2}(X(P_0+P_0^{\#}) - (P_0+P_0^{\#})X);$$ 
and  
$$
\frac{1}{2i}(X(P_0-P_0^{\#}) - (P_0-P_0^{\#})X),$$ 
respectively. Clearly,  $T_{Q_0}(\uj \cdot Q_0)$  will be complemented in $L(\h)$ if 
\[ 
\mathcal{L}_s:=\{  \, XE_0 - E_0X  + \frac{1}{2}(X(P_0+P_0^{\#}) - (P_0+P_0^{\#})X)  \, : \, X^\#=-X \, \}
\] 
is complemented in the subspace of $J-$selfadjoint operators  $i\lie$ and
\[
\mathcal{L}_a:=\{  \,  X(P_0-P_0^{\#}) - (P_0-P_0^{\#})X \, : \,  X^\#=-X \, \}
\] 
is complemented in $\lie$.

\medskip

\begin{lem}\label{sup self}
 $\mathcal{L}_s$ is complemented in $i \lie$. 
\end{lem}

\begin{proof}
Set $\s_1=R(E_0)$,  $\s_2=R(P_0+P_0^\#)$ and $\s_3=R(F_0)$ and consider the $J$-orthogonal decomposition $\h=\s_1 \sdo \s_2 \sdo \s_3$.
If a $J$-antihermitian operator $X$ is represented as a block-operator matrix according to this decomposition:
\[
  \begin{blockarray}{lrrrl}
    \begin{block}{l(rrr)l}
  \, & X_{11} &   X_{12} &  X_{13} \, & \matindex{$\s_1$} \\
  \, X=&-X_{12}^{\#} &  X_{22}  &  X_{23} \, & \matindex{$\s_2$} \\
   \, &-X_{13}^{\#}  &  -X_{23}^{\#} &  X_{33}\, &  \matindex{$\s_3$} \\
    \end{block}
  \end{blockarray},
\]
then an operator in  the subspace $\mathcal{L}_s$ 
is represented as
$$XE_0 - E_0X  + \frac{1}{2}(X(P_0+P_0^{\#}) - (P_0+P_0^{\#})X)=\left( \begin{array}{ccc}
0 &   -\frac{1}{2}X_{12} &  -\frac{1}{2}X_{13} \\ 
-\frac{1}{2}X_{12}^{\#} &  0  &  -\frac{1}{2}X_{23}  \\
 -\frac{1}{2}X_{13}^{\#}  &  -\frac{1}{2}X_{23}^{\#} &  0 \\
\end{array}
\right).$$
Therefore, the subspace $\mathcal{L}_s$ can be described with three parameters as
\[ \mathcal{L}_s=
\left\{ \, 
\left( \begin{array}{ccc}
0 &   A & B  \\ 
A^{\#} &  0  &  C  \\
 B^{\#}  &  C^{\#} &  0 \\
\end{array}
\right)
\, : 
\, A \in L(\s_2, \s_1), \, B \in L(\s_3, \s_1), \, C \in L(\s_3, \s_1) \,  
\right\}.
\]
From this representation, it is easy to see that $\mathcal{L}_s$ is complemented in the subspace of $J$-selfadjoint operators. In fact, a complement is given by the subspace of $J$-selfadjoint operators which are  block-diagonal according to the decomposition considered above. 
\end{proof}

As in the previous result, the main idea in the proof of the following lemma is to find an alternative description of $\mathcal{L}_a$ by means of $3 \times 3$ block-operator matrices. 
However, the decomposition will be given in terms of different projections. 

\begin{lem}\label{sup antih}
$\mathcal{L}_a$ is complemented in $ \lie$. 
\end{lem}

\begin{proof}
Set $A_0=P_0-P_0^{\#}$. Note that $A_0^2=P_0+P_0^{\#}$ and $A_0^3=A_0$. In particular, $A_0 ^2$ is a $J$-selfadjoint 
projection. Now set
\[  R_0=\frac{1}{2}(A_0 ^2 + A_0). \]
From the properties of $A_0$, it follows that 
$$R_0^2=\frac{1}{4}(A_0^4 + 2 A_0 ^3 + A_0 ^2)=\frac{1}{4}(2A_0^2   +   2A_0)=R_0,$$ 
hence $R_0$ is a projection. Then $R_0^{\#}$ is also a projection.   Taking into account that $A_0^{\#}=-A_0$, one gets 
$R_0 + R_0^{\#}=A_0^2$. Further useful relations between these projections are the following:
\[ \, \, \, A_0R_0=R_0, \, \, \ R_0A_0=R_0 , \, \, \, -R_0^{\#} A_0=R_0^{\#}, \, \, \, - A_0R_0^{\#}=R_0^{\#}\, . \]
Set $\t_1=R(R_0)$, $\t_2=R(R_0^{\#})$ and $\t_3=R(I-A_0 ^2)$. Hence the space $\h$ can be decomposed as $\h=\t_1\dotplus \t_2\dotplus \t_3$.
Next, suppose that $X=-X^{\#}$ is represented as
\[
  \begin{blockarray}{lrrrl}
    \begin{block}{l(rrr)l}
  \, &  X_{11} &   X_{12} &  X_{13} \, & \matindex{$\t_1$} \\
  \, X= & X_{21} &  X_{22}  &  X_{23} \, & \matindex{$\t_2$} \\
   \, & X_{31}  &  X_{32} &  X_{33}\, &  \matindex{$\t_3$} \\
    \end{block}
  \end{blockarray}
\]
From the properties of the projections $R_0$, $R_0^{\#}$ and $I-A_0^2$,  it is possible to consider only five parameters to represent $X$ as a block operator-matrix, that is
$$X=\left( \begin{array}{ccc}
X_{11} &   X_{12} &  X_{13} \\ 
-X_{12}^{\#} &  -X_{11}^{\#}  &  X_{23}  \\
 -X_{23}^{\#}  &  -X_{13}^{\#} &  X_{33} \\
\end{array}
\right),$$
\noi where $X_{33}^{\#}=-X_{33}$. On the other hand, the operator $XA_0-A_0X \in \mathcal{L}_a$ is represented as
\[ XA_0-A_0X = 
\left( \begin{array}{ccc}
0 &   2X_{12} &  X_{13} \\ 
-2X_{12}^{\#} &  0  &  X_{23}  \\
 -X_{23}^{\#}  &  -X_{13}^{\#} &  0 \\
\end{array}
\right).
\]
Therefore, 
$$\mathcal{L}_a=\bigg\{ \, 
 \left( \begin{array}{ccc}
0 &   A & B  \\ 
-A^{\#} &  0  &  C  \\
 -C^{\#}  &  -B^{\#} &  0 \\
\end{array}
\right)  \, : 
\, A \in L(\t_2 , \t_1),\, B \in L(\t_3 , \t_1),\, C \in L(\t_3 , \t_2) \, \bigg\}. $$
Now the subspace $\mathcal{L}_a$ can be easily complemented in $\lie$ as follows:
\[    \lie = \mathcal{L}_a \oplus \bigg\{  \,     
\left( \begin{array}{ccc}
Y &   0 & 0  \\ 
0 &  -Y^{\#}  &  0  \\
 0  &  0 &  Z \\
\end{array}
\right)
\, : \, Y \in L(\t_1) ,\, Z \in L(\t_3), \, Z=-Z^{\#} \, \bigg\}. \qedhere        \]
\end{proof}

\noi The main facts on the differential structure of $\q$ and $\e$ are collected in the following result.

\begin{teo}\label{q submanifold lh}
The following assertions hold:
\begin{enumerate}
\item[i)] $\q$ is an analytic submanifold of $L(\h)$.
\item[ii)] $\e$ is an analytic submanifold of $L(\h)$.
\item[iii)] The map $F:\q \to \e$,  $F(Q)=QQ^\#$,  is a real analytic submersion.
\end{enumerate}
\end{teo}
\begin{proof}
${\rm i)}$ The assumptions in the criterion stated in Proposition \ref{sous varietes} are verified in each connected component of $\q$. Indeed, it has been shown in Theorem  \ref{section} that the quotient topology of $\uj \cdot Q_0$ coincides with the topology inherited from $L(\h)$. In addition, the tangent space $T_{Q_0}(\uj \cdot Q_0)$ is complemented in $L(\h)$ by Lemmas \ref{sup self} and \ref{sup antih}. 
But this says that the range of the differential map of the inclusion $\uj \cdot Q_0 \hookrightarrow L(\h)$ is complemented in $L(\h)$. So, the proof is  completed. 

\smallskip

\noi ${\rm ii)}$ It is analogous to the proof of ${\rm i)}$.

\smallskip 

\noi ${\rm iii)}$ It suffices to prove the statement for a connected component  $\uj \cdot Q_0$  of $\q$ and  a connected component $\uj \cdot E_0$ of $\e$, where $E_0=Q_0Q_0^{\#}$. According to Proposition \ref{sous varietes}, the identity map $\uj \cdot Q_0 \simeq \uj / \mathcal{G} \to \uj \cdot Q_0 \subseteq L(\h)$ is bianalytic. Thus, if one considers the submanifold structure in $\uj \cdot Q_0$, then the map $p_{Q_0}$ is also an analytic submersion. Analogously, the map $p_{E_0}$ is an analytic submersion when this orbit has the submanifold structure. 

Next note that the following diagram commutes 
\begin{displaymath}
\xymatrix{
\uj \ar[r]^{p_{Q_0}} \ar[dr]_{p_{E_0}} & \uj \cdot Q_0 \ar[d]^{F}  \hspace{-0.5cm}& \hspace{-.5cm}  \\
                                   & \uj \cdot E_0             \hspace{-0.5cm} &  \hspace{-0.5cm}  
}
\end{displaymath}
Since $p_{Q_0}$ is a surjective analytic submersion and $p_{E_0}$ is an analytic submersion, it is a well-known fact that $F$ turns out to be an analytic submersion (see for instance \cite[Corollary 8.4]{upmeier85}).
\end{proof}

\section{Covering space structure of $\q_{\s}$}

For a fixed pseudo-regular subspace $\s$, consider the set of $J$-normal projections onto $\s$, i.e.
\[
\Q_{\s} =\{Q\in \Q:\, R(Q)=\s\}. 
\]
Clearly, the group $\uj$ does not leave $\q_{\s}$ invariant. In order to find a suitable group acting on $\q_{\s}$, one can restrict to
the subgroup of $\uj$ given by
\[ \u_{\s}=\{ \, U \in \uj \, : \,  U(\s)=\s \, \}.  \]
It is easy to see that if $U\in \u_\s$ then $U(\s^{\ort})=\s^{\ort}$ and $U(\s^{\circ})=\s^{\circ}$.

\medskip
As before, the action of $\u_\s$ on $\q_\s$ 
is defined by $U \cdot Q=UQU^{\#}$, where $U \in \u_{\s}$ and $Q \in \q_{\s}$.

The following result was proved in \cite[Proposition 3.1]{G}. Below there is another proof with an explicit 
construction of the $J$-isometric isomorphism. This formula will be helpful later. Along this section, when  $\t_1$, $\t_2$ are two (closed) subspaces of $\s$ such that $\t_1 \dot{+} \t_2=\s$, the projection in $L(\s)$  with range $\t_1$ and nullspace $\t_2$ is denoted by $P_{\t_1 // \t_2}$.

\begin{lem}\label{j isom iso}
Let $\s$ be a pseudo-regular subspace of $\h$. If $\m_1$, $\m_2$ are two regular subspaces such that $\s=\m_1 [\dot{+}] \s^{\circ}=\m_2 [\dot{+}] \s^{\circ}$, then 
$(P_{\m_2//\s^{\circ}} )|_{\m_1}: \m_1 \to \m_2$ is a $J$-isometric isomorphism.
\end{lem}
\begin{proof}
Set $W=(P_{\m_2//\s^{\circ}} )|_{\m_1}\in L(\m_1, \m_2)$. Let $f \in \m_1$ such that $W f=0$. Then, $f \in \s^{\circ}\cap \m_1=\{0\}$. Thus, $W$ is one-to-one. To show that $W$ is surjective, pick $g \in \m_2$. Then $g=f_{\m_1}+f_{\s^{\circ}}$, where $f_{\m_1} \in \m_1$ and $f_{\s^{\circ}} \in \s^{\circ}$. Therefore $g=P_{\m_2//\s^{\circ}}g=P_{\m_2//\s^{\circ}}f_{\m_1} $. Hence $g=Wf_{\m_1}$. 

Finally, notice that $W$ is a $J$-isometric isomorphism. Indeed, given $f,g \in \m_1$, suppose that $f=f_{\m_2}+f_{\s^{\circ}}$ and $g=g_{\m_2}+g_{\s^{\circ}}$. Since $f_{\s^{\circ}}, g_{\s^{\circ}} \in  \s^{\circ}$, it follows that
\[  [Wf,Wg]=[f_{\m_2},g_{\m_2}]=[f,g]. \]
Hence $W$ is $J$-isometric.
\end{proof}

The next result shows that, given $Q_0\in\q_\s$, any other $Q\in\q_\s$ can be written as $Q=UQ_0U^\#$ for a suitable $U\in \u_\s$.

\begin{prop}\label{accion t grup chic}
The group $\u_{\s}$ acts transitively on $\Q_{\s}$.
\end{prop}
\begin{proof}
 If $Q, Q_0\in \Q_\s$, consider the usual associated  projections $E$, $F$, $P$ and $E_0$, $F_0$, $P_0$. According to Remark \ref{desc q y 1-q},  $\h$ can be decomposed as
\[
 \h= R(E)\sdo R(P + P^\#) \sdo R(F)= R(E_0)\sdo R(P_0 + P_0^\#) \sdo R(F_0).
\]
Notice that $R(P)=R(P_0)=\s^\circ$. 
Then, by Lemma \ref{s link}, there exists a $J$-isometric isomorphism 
$$V: R(P_0 + P_0^\#) \ra R(P + P^\#),$$
which can be defined as the identity operator on $\s^{\circ}$. 
On the other hand, by Lemma \ref{j isom iso}, there is a $J$-isometric isomorphism $W: R(E_0) \ra R(E)$.

It only remains to show that the ranges of $F$ and $F_0$ are $J$-isometrically isomorphic.
To this end, note that $\s^{[\perp]}$ is also a pseudo-regular subspace.  Moreover, it follows that $\s^{[\perp]}=R(Q)^{[\perp]}=N(Q^{\#})=R(I-Q^{\#})=R(F+P)= R(F)\sdo \s^{\circ}$. Similarly,  one can see that $\s^{[\perp]}=R(F_0)[\dot{+}] \s^{\circ}$. 
Therefore, $R(F)$ and $R(F_0)$ are two different regular complements of $\s^{\circ}$ in $\s^{[\perp]}$. As in the previous paragraph, there is a $J$-isometric isomorphism $W': R(F_0) \ra R(F)$.

Finally, define $U: \h\ra\h$ by $U(f+g+h)=Wf + Vg + W'h$, where $f\in R(E_0)$, $g\in R(P_0 + P_0^\#)$ and $h\in R(F_0)$.
It is easy to see that $U \in \u_{\s}$ and, by construction,  $UQ_0U^\#=Q$.
\end{proof}

Given a pseudo-regular subspace $\s$ of $\h$, consider the family of regular complements of $\s^\circ$ in $\s$:
 $$\f=\{ \, \m \text{ is a regular subspace of $\h$} \, : \, \s=\m [\dot{+}] \s^{\circ}  \, \}.$$
It is not difficult to see that  $\q_{\s}$  can be rewritten as the following disjoint union
\begin{equation}\nonumber
\q_\s= \bigcup_{\m \in \f} \q_{\s,\M},
\end{equation}
where 
$ \q_{\s,\M}=\{ \, Q \in \q_{\s} \, : \, R(QQ^{\#})=\m \,  \}$, see \cite[Lemma 6.4]{maestripieri martinez13}.
For each $\m\in \f$, it is natural to consider the subgroup $\u_{\m, \s^{\circ}}$ of $\u_{\s}$ defined by
\[ \u_{\m, \s^{\circ}}=\{ \, U \in \uj \, : \,  U(\m)=\m, \,  U(\s^{\circ})=\s^{\circ} \, \}.  \]
Clearly, $\u_{\m, \s^{\circ}}$ acts on $\q_{\s , \m}$ by conjugation. Furthermore,

\begin{prop}\label{accion t grup chic deck}
The group $\u_{\m, \s^{\circ}}$ acts transitively on $\Q_{\s, \m}$.
\end{prop}
\begin{proof}
The same proof of Proposition \ref{accion t grup chic} works in this case. Indeed, the $J$-unitary $U$ constructed in that proof leaves $\m$ invariant whenever $Q,Q_0 \in \q_{\s, \m}$.
\end{proof}

In the next result a continuous selection from $\f$ to $\q_{\s}$ is constructed. The set $\f$ is endowed with the topology defined by the metric $$d(\m , \n)=\| E_{\m} - E_{\n} \|,$$ where $E_{\m}$ denotes the (unique) $J$-selfadjoint projection onto $\m$;
 meanwhile $\q_{\s}$ is considered with the topology inherited from $L(\h)$.

\begin{lem}\label{cont sel}
There exists a continuous map 
$g:\f \to \q_{\s}$ such that $g(\m) \in \q_{\s , \m}$.
\end{lem}
\begin{proof}
 Let $\m$ be a regular subspace of 
$\h$ such that $\s=\m [\dot{+}] \s^{\circ}$. Consider the following orthogonal decomposition  $\h=\s^{\circ} \oplus (\s \ominus \s^{\circ} )\oplus \s^{\perp}$. According to \cite[Theorem 6.9]{maestripieri martinez13}, a $J$-normal projection $Q$ belongs to $\q_{\s, \m}$ if and only if $Q$ can be written as 
\[ 
Q=\left( \begin{array}{ccc}
I &  0 & A + (\text{Re}(Bc^*br a^*) - \frac{1}{2}(BdB^*+ar^*b^3ra^*))a + B + ar^*(c+b) \\
 0 &  I  &  b^{-1}c + r \\
 0  &  0 &  0 
\end{array}
\right),
\]
where $r=P_{\s \ominus \s^{\circ}} E_{\m}|_{J(\s^\circ)}\in L(\s^\bot, \s\ominus\s^\circ)$, $A=-A^* \in L(\s^{\circ})$ and $B \in L(\s^{\perp}, \s^{\circ})$ satisfies $J(\s^{\circ})\subseteq N(B)$. Here the lowercase letters $a,b,c$ and $d$ come from the decomposition of the fundamental symmetry
\[ 
	J= \left( \begin{array}{ccc}
	 0 &  0 & a \\
	0 &  b  &  c \\
	a^*  &  c^* &  d
	\end{array}
\right).
\]
Clearly, $r:\f \to L(\s^\bot, \s\ominus\s^\circ)$ given by $r(\m)=P_{\s \ominus \s^{\circ}} E_{\m}|_{J(\s^\circ)}$ is a continuous function by the definition of the metric in $\f$.
To construct the required continuous selection, it is possible to set $A=B=0$ in the above decomposition. Therefore, the map
$$
g: \f \to \q_{\s} \, \, \, \text{defined by}\, \, \, 
g(\m)=\left( \begin{array}{ccc}
I &   0 & - \frac{1}{2}ar^*b^3ra^*a + ar^*(c+br) \\ 
0 &  I  & b^{-1}c + r  \\
 0  &  0 &   0\\
\end{array}
\right),
$$
 satisfies $g(\m) \in \q_{\s , \m}$. The continuity of $g$ follows from that of $r$.
\end{proof}

The map defined locally in Lemma \ref{sec cont lema} can be defined globally in $\q_{\s}$. This allows to prove the existence of a global section for the restriction of $p_{Q_0}$ to $\u_\s$.

\begin{prop}\label{sec global qs}
Let $Q_0$ be a projection in $\q_\s$. Then, the map $$p_{Q_0}: \u_{\s} \to \q_\s, \, \, \, \, 
p_{Q_0}(U)=UQ_0U^{\#},$$
has a global continuous cross section.
\end{prop}

\begin{proof}
First recall that $\h = R(E_i) \sdo R(F_i)\sdo R(P_i + P_i^\#)$ for $i =1,2$ and consider $U:\q_{\s} \times \q_{\s} \to \u_{\s}$ defined by
\[
U(Q_1,Q_2)=P_{R(E_2)//\s^{\circ}}  E_1  + P_{R(F_2)//\s^{\circ}}  F_1  + V(Q_2)V(Q_1)^{\#}(P_1 + P_1^{\#}).
\]
The map $U$ is a $J$-isometric isomorphism 
restricted to each of these three pairs of subspaces. In fact,  the map $R(E_1) \to R(E_2)$ given by $f \mapsto  P_{R(E_2)//\s^{\circ}}f$ is a $J$-isometric isomorphism by Lemma \ref{j isom iso}. Similarly, one can see that $R(F_1) \to R(F_2)$ given by $f \mapsto  P_{R(F_2)//\s^{\circ}}f$ is also a $J$-isometric isomorphism. 
Also, by Lemma \ref{s link}, $V(Q_2)V(Q_1)^\#$ is a $J$-isometric isomorphism from $R(P_1 + P_1^\#)$ onto $ R(P_2 + P_2^\#)$. 
Hence $U(Q_1,Q_2)$ is a $J$-unitary. Moreover, it is the identity operator in $\s^{\circ}$, and it leaves $\s$ invariant. Hence the operator $U(Q_1,Q_2)$ belongs to $ \u_{\s}$. In addition, by Lemma \ref{j isom iso} and Lemma \ref{s link} it follows that 
\[
U(Q_1,Q_2) \, Q_1 \, U(Q_1,Q_2)^{\#}=Q_2.
\]
Note that $E_i=Q_i Q_i ^{\#}$, $F_i=(I-Q_i)(I- Q_i )^{\#}$ and $P_i=Q_i(I-Q_i ^{\#})$ are continuous functions of $Q_i$ for $i=1,2$. On the other hand, the continuity of $V(Q_i)$, $i=1,2$, and consequently, the continuity of $ V(Q_2)V(Q_1)^{\#}$, is proved in
 Lemma \ref{s link}. To show that $U$ is a continuous map, it remains to prove that the maps
$\q_\s \to L(\s)$  given by $Q \mapsto P_{R(E)//\s^{\circ}}$ and $\q_\s \to L(\s^{\ort})$  given by $Q \mapsto P_{R(F)//\s^{\circ}}$ are continuous. 

Let $\{ Q_k \}_{k \geq 1}$ be a sequence in $\q_{\s}$ such that $\| Q_k - Q\| \to 0$. 
Again recall that $E=Q Q ^{\#}$ is a continuous function of $Q$.  
Thus, by Lemma \ref{kato proj} one finds that 
\[   \| P_{R(E_k)} - P_{R(E)} \| \leq \| E_k - E\| \to 0.  \]  
Applying the formula proved in \cite[Lemma 3.1]{ando11}, the projection $ P_{R(E) // \s^\circ}$ can
be rewritten as
\[   P_{R(E) // \s^\circ}=P_{R(E)}(P_{R(E)}+P_{\s^\circ})^{-1}.   \] 
Notice that this formula is a continuous function of the orthogonal projection $P_{R(E)}$. Hence it
follows that $\|P_{R(E_k) // \s^\circ} - P_{R(E) // \s^\circ} \| \to 0$. The proof of the continuity of $Q \mapsto P_{R(F)//\s^{\circ}}$ is similar. 

Therefore, the map $s:\q_\s \to \u_{\s}$ defined by $s(Q)=U(Q_0,Q)$ is a global continuous cross section of $p_{Q_0}(U)=UQ_0U^{\#}$.
\end{proof}

In the next result $s$ stands for the global section considered in Proposition \ref{sec global qs}, and $g$ is the continuous selection defined in Lemma \ref{cont sel}.

\begin{teo}\label{cov spa fixed range}
Let $\s$ be a pseudo-regular subspace and $\m_0$ be a regular subspace such that 
$\s=\m_0 \sdo \s^{\circ}$.  Let  $Q_0$ be a fixed projection in $\q_{\s , \m_0}$. Consider the map $ r: \q_\s \to \q_{\s , \m_0}$ defined by
\[
r(Q)=s(g(\m))^{\#}\, Q \,s(g(\m)), 
\]
whenever $Q \in \q_{\s , \m}$. Then $r$ is a covering map.
\end{teo}

The map $r$ has an alternative expression. By Lemma \ref{accion t grup chic deck} there exists $U \in \u_{\m,\s^{\circ}}$ such that $Q=Ug(\m)U^{\#}$. Therefore, note that
\begin{align}
r(Q) &= s(g(\m))^{\#} \, U g(\m) U^{\#}  \, s(g(\m)) \nonumber \\
 &= s(g(\m))^{\#} U s(g(\m)) \, Q_0 \, s(g(\m))^{\#} U^{\#} s(g(\m))\nonumber \\
&= Ad_{s(g(\m))} (U)\, Q_0 \, (Ad_{s(g(\m))} (U))^{\#}, \label{another exp}
\end{align}
where $Ad_U: L(\h) \to L(\h)$ is defined by $Ad_U(X)=U^{\#}XU$, for $U \in \uj$ and $X \in L(\h)$. This expression does not depend on the choice of $U\in\u_{\m,\s^\circ}$.

\begin{proof}
It has been  previously noted that $\q_\s$ is the disjoint union of the decks $\q_{\s,\m}$ with $\m \in \f$. Then, for any 
$Q \in \q_\s$ there exists a unique $\m \in \f$ such that $Q \in \q_{\s, \m}$. Thus, $r$ is well defined. 

Next notice that $r$ is a surjective map. For this purpose it is helpful to use the alternative expression of $r$. Suppose that 
\[
g(\m)=E_\m + P \ \ \ \text{ and } \ \ \ I-g(\m)=F+ P^{\#}. 
\]
Note that $V=s(g(\m))$ satisfies $V(\m_0)=\m$, $V(R(F_0))=R(F)$, $V(\s^{\circ})=\s^{\circ}$ and $V(R(P_0^{\#}))=R(P^{\#})$. From this latter fact, it is not difficult to see that  $Ad_{V}(\u_{\m, \s^{\circ}})=\u_{\m_0, \s^{\circ}}$. 
According to Lemma  \ref{accion t grup chic deck} the group $\u_{\m_0, \s^{\circ}}$ transitively acts on $\q_{\s ,\m_0}$,   and consequently,  $r$ turns out to be surjective. 

Observe that there is a continuous inverse of the restriction of $r$ to $\q_{\s ,\m}$, which is given by
$$
f:=(r|_{\q_{\s ,\m}})^{-1}: \q_{\s ,\m_0} \to \q_{\s ,\m}, \, \, \, f(Q)=s(g(\m)) \, Q\, s(g(\m))^{\#} .
$$
Since $s$ is a continuous map by Proposition \ref{sec global qs}, it follows that $f$ is continuous. To show that $f$ is actually  the inverse of $r|_{\q_{\s ,\m}}$, observe that if $Q\in\q_{\s,\m}$,
\begin{align*}
(f\circ r)(Q) &= f(s(g(\m))^{\#}\, Q \,s(g(\m))) \\
&= s(g(\m)) s(g(\m))^{\#}\, Q \,s(g(\m)) s(g(\m))^{\#}=Q.
\end{align*}
Also, $f(Q) \in \q_{\s ,\m}$ and
\begin{align*} 
(r\circ f)(Q) &=r(s(g(\m)) \, Q\, s(g(\m))^{\#})= \\
&= s(g(\m))^{\#} s(g(\m)) \, Q\, s(g(\m))^{\#} s(g(\m))=Q.
\end{align*} 
The map $g$ is continuous by  Lemma \ref{cont sel}, meanwhile 
the map $s$  is continuous by Proposition \ref{sec global qs}. 
Thus,  $r$ is clearly a continuous map. This completes the proof. 
\end{proof}

\subsection*{Acknowledgement}

The authors are indebted with Prof. Aurelian Gheondea for some discussions on the material of this manuscript. 
In addition, they  wish to thank the referee for his/her careful reading of this work and helpful comments.

{\bf Eduardo Chiumiento}

Facultad de Ciencias Exactas -- Universidad Nacional de La Plata

and

Instituto Argentino de Matem\'atica ``Alberto P. Calder\'on''

Saavedra 15, Piso 3

(1083) Buenos Aires 

Argentina

eduardo@mate.unlp.edu.ar

\medskip

{\bf Alejandra Maestripieri}

Facultad de Ingenier\'{\i}a -- Universidad de Buenos Aires

and 

Instituto Argentino de Matem\'atica ``Alberto P. Calder\'on''

Saavedra 15, Piso 3

(1083) Buenos Aires 

Argentina 

amaestri@fi.uba.ar

\medskip

{\bf Francisco Mart\'{\i}nez Per\'{\i}a}

Facultad de Ciencias Exactas -- Universidad Nacional de La Plata

and

Instituto Argentino de Matem\'atica ``Alberto P. Calder\'on''

Saavedra 15, Piso 3

(1083) Buenos Aires 

Argentina 

francisco@mate.unlp.edu.ar

\end{document}